\documentclass[12pt, reqno]{amsart}
 \newcommand{\todo}[1]{}

\makeatletter
\DeclareMathAlphabet{\mathpzc}{OT1}{pzc}{m}{it}
\usepackage[margin=2.4cm]{geometry}
\usepackage{amsmath, amssymb, amsfonts, amstext, verbatim, amsthm, mathrsfs}
\usepackage[mathcal]{eucal}
\usepackage{microtype}

\DeclareSymbolFont{bbold}{U}{bbold}{m}{n}
\DeclareSymbolFontAlphabet{\mathbbold}{bbold}

\usepackage[usenames]{color}
\usepackage{aliascnt} 
\usepackage[colorlinks=true,linkcolor=blue,citecolor=blue,urlcolor=blue,citebordercolor={0 0 1},urlbordercolor={0 0 1},linkbordercolor={0 0 1}]{hyperref} 
\usepackage{enumerate}
\usepackage{xspace}

\theoremstyle{plain}

\newcommand{\refnewtheoremn}[4]{
\newaliascnt{#1}{#2}
\newtheorem{#1}[#1]{#3}
\aliascntresetthe{#1}
\expandafter\providecommand\csname #1autorefname\endcsname{#4}}

\newcommand{\refnewtheorem}[3]{\refnewtheoremn{#1}{#2}{#3}{#3}}

\def\makeCal#1{
\expandafter\newcommand\csname c#1\endcsname{\mathcal{#1}}}
\def\makeBB#1{
\expandafter\newcommand\csname b#1\endcsname{\mathbb{#1}}}
\def\makeFrak#1{
\expandafter\newcommand\csname f#1\endcsname{\mathfrak{#1}}}

\count@=0
\loop
\advance\count@ 1
\edef\y{\@Alph\count@}
\expandafter\makeCal\y
\expandafter\makeBB\y
\expandafter\makeFrak\y
\ifnum\count@<26
\repeat

\newtheorem{theorem}{Theorem}[section]

\refnewtheorem{lemma}{thm}{Lemma}
\refnewtheorem{assumptions}{thm}{Assumptions}
\refnewtheorem{examples}{thm}{Examples}
\refnewtheorem{claim}{thm}{Claim}
\refnewtheorem{cor}{thm}{Corollary}
\refnewtheorem{conj}{thm}{Conjecture}
\refnewtheorem{prop}{thm}{Proposition}
\refnewtheorem{proposition}{thm}{Proposition}
\refnewtheorem{quest}{thm}{Question}
\refnewtheorem{assumption}{thm}{Assumption}

\theoremstyle{definition}
\refnewtheorem{remark}{thm}{Remark}
\refnewtheorem{remarks}{thm}{Remarks}
\refnewtheorem{defn}{thm}{Definition}
\refnewtheorem{definition}{thm}{Definition}
\refnewtheorem{notn}{thm}{Notation}
\refnewtheorem{const}{thm}{Construction}
\refnewtheorem{example}{thm}{Example}
\refnewtheorem{fact}{thm}{Fact}
\refnewtheorem{problem}{thm}{Problem}

\newcommand {\id}{\operatorname{id}}

\newcommand{\Stab}{\operatorname{Stab}}

\newcommand {\Hom}{\operatorname{Hom}}
\newcommand{\End}{\operatorname{End}}
\newcommand {\Aut}{\operatorname{Aut}}
\newcommand{\DT}{\operatorname{DT}}
\newcommand{\Ham}{\operatorname{Ham}}

\newcommand {\<}{\langle}
\renewcommand {\>}{\rangle}
\newcommand{\half}{\tfrac{1}{2}}
\newcommand{\tensor}{\otimes}
\renewcommand{\O}{\mathscr{O}}
\newcommand{\Lie}{\cL}
\newcommand{\isom}{\cong}
\newcommand{\mat}[4]{\begin{pmatrix}#1&#2\\#3&#4\end{pmatrix}}

\newcommand{\one}{\mathbbold {1}}
\newcommand{\hk}{hyperk{\"a}hler }

\newcommand{\fg}{\mathfrak g}
\newcommand{\fh}{\mathfrak h}
\newcommand{\fgod}{\fg^{\rm od}}
\newcommand{\GL}{\operatorname{GL}}
\newcommand{\hreg}{\fh^{{\rm reg}}}
\newcommand{\D}{\cD}
\newcommand{\sdiff}{\rm{sdiff}}

\begin{document}
\title{Complex hyperk{\"a}hler structures defined by Donaldson-Thomas invariants}
\author{Tom Bridgeland} \author{Ian A. B. Strachan}

\date{}

\begin{abstract}
The notion of a Joyce structure was introduced in \cite{RHDT2} to describe the geometric structure on the space of stability conditions of a CY$_3$ category naturally encoded by the Donaldson-Thomas invariants.
In this paper we show that a Joyce structure on a complex manifold defines a complex \hk structure on the total space of its tangent bundle, and  give a characterisation of the resulting \hk metrics in geometric  terms. 
\end{abstract}

\maketitle

\section{Introduction}

In the recent paper \cite{RHDT2} it was argued that the Donaldson-Thomas (DT) invariants  of a CY$_3$ triangulated category  $\cD$ should encode a certain geometric structure on its space of stability conditions $M=\Stab(\D)$. These structures were called Joyce structures, since the most important ingredients already appear in the paper \cite{HolGen}.  The definition of a Joyce structure on a complex manifold $M$ given in \cite{RHDT2} is rather ungeometric in nature. In this paper we explain that it can be re-interpreted as the existence of a complex \hk metric of a particular kind on the total space $X=\cT_M$ of the holomorphic tangent bundle of $M$. We include a brief summary of the most important ideas of \cite{RHDT2}  in the Appendix.

In terms of a  local system of co-ordinates $(z_1, \cdots, z_n)$ on  $M$, the main   ingredient of a Joyce structure is
a  function\footnote{This was denoted by the letter $J$ in \cite{RHDT2}, but to avoid confusion with the standard notation $I,J,K$ in \hk geometry, we switch to $W$ here.} $W\colon X\to \bC$ satisfying the partial differential equation \begin{equation}
\label{point_intro}
\frac{\partial^2 W}{\partial \theta_i \partial z_j}-\frac{\partial^2 W}{\partial \theta_j \partial z_i }=\sum_{p,q} \eta^{pq} \cdot \frac{\partial^2 W}{\partial \theta_i \partial \theta_p} \cdot \frac{\partial^2 W}{\partial \theta_j \partial \theta_q},\end{equation}
where $\eta^{pq}$ is a constant non-degenerate skew-symmetric matrix, and  $(\theta_1,\cdots,\theta_n)$ are the natural linear co-ordinates on the tangent spaces $\cT_{M,p}$ obtained by writing a tangent vector in the form $v=\sum_{i=1}^n \theta_i\cdot  \frac{\partial}{\partial z_i}$. We shall explain below that \eqref{point_intro} is the condition for the expression
\begin{equation}
\label{met}g=\sum_{i,j} \omega_{ij}  \cdot (d \theta^i \tensor dz^j+dz^j\tensor d \theta^i  )-\sum_{i,j} \frac{\partial^2 W}{\partial \theta_i\partial \theta_j} \cdot (dz^i\tensor dz^j+dz^j\tensor dz^i).\end{equation}
to define a complex \hk metric on $X$. Here $\omega_{ij}$ is the inverse matrix to $\eta^{ij}$. In terms of the basis of vector fields
\[
v_i=\frac{\partial}{\partial \theta_i},\qquad h_i= \frac{\partial}{\partial z_i} + \sum_{p,q} \eta^{pq} \cdot \frac{\partial^2 W}{\partial \theta_i \partial \theta_p} \cdot \frac{\partial}{\partial \theta_q},    \]
the complex structures $I,J,K$ are defined by the block matrices
\[I=\mat{i\cdot \one}{0}{0}{-i\cdot \one}, \qquad J=\mat{0}{-\one}{\one}{0}, \qquad K=\mat{0}{-i\cdot \one}{-i\cdot \one}{0},\]and the metric is given by $g=\sum_{i,j} \omega_{ij} \cdot (v^i\tensor h^j+h^j\tensor v^i )$.

The main characteristic of the  complex \hk structures defined by this construction is that the holomorphic 2-form
\[\Omega_-(v,w)=g(v,(J-iK)(w)),\]
on $X$ is the pull-back via the natural projection $\pi\colon X\to M$ of a holomorphic symplectic form on $M$, namely $\omega=\sum_{i,j} \omega_{ij} \, dz_i\wedge dz_j$.  Conversely, as we  show in Section 2,  all complex \hk structures with this property are locally given by the above construction for some function $W$ satisfying \eqref{point_intro}.

In a recent preprint \cite{Dunajski}, Dunajski showed that the form (\ref{met}) of the metric follows from a weaker set of conditions than complex \hk, namely null K\"ahler. Such a structure consists of a metric $g$ and an endomorphism $N$ with the property that its kernel is half-dimensional, together with the conditions
\[
g(X,NY)+g(NX,Y)=0\,,\qquad N^2=0
\]
and $\nabla N=0$ (with $\nabla$ being the Levi-Civita connection of $g$). Since $(J-iK)^2=0$ the constructions in this paper provides examples of such null-K\"ahler structures. This work also provides an explanation of the appearance of isomonodromy problems.

\subsection{Relations with previous work}

Since \hk metrics may be written in terms of a K{\"a}hler potential, the geometric conditions of being \hk result in differential equations for such a potential, and these take the form of Monge-Amp\'ere-type equations. In 4-dimensions there is a single equation which is known as Pleba\~nski's first heavenly equation \cite{Pleb} (the original motivation coming from the equivalent description of Ricci-flat metrics with anti-self-dual Weyl tensor). Many other forms of the equations exist, and equation (\ref{point_intro}), in 4-dimensions, is known as Pleba\~nski's second heavenly equation. From the work of Penrose \cite{Pen} - the original nonlinear graviton construction - it immediately follows that there is an associated twistor space, and a crucial property of this space is the family of curves with normal bundle ${\mathcal O}(1)\oplus {\mathcal O}(1)\,.$

This construction can be generalized in many ways, and the simplest is to generalize the normal bundle structure to ${\bigoplus}_{i} {\mathcal O}(n_i)$. Equation (\ref{point_intro}) first appear {\sl explicitly} in the literature in the work of Takasaki \cite{Tak}, who used the bundle of 2-forms construction of Gindikin \cite{Gin} to write down the associated hierarchies of integrable equations. Equation (\ref{point_intro}) corresponds to the case where the manifold is \hk, with the associated family of curves in the twistor space having normal bundle structure ${\mathcal O}(1)\oplus \ldots \oplus {\mathcal O}(1)$ with $2r$-terms.

With the development of the links between twistor theory and the theory of integrable systems initiated by Ward \cite{Ward}, these curved twistor space constructions and their pencils of commuting vector fields were reinterpreted in terms of Lax equations with \lq gauge fields\rq~taking values in, for example, the Lie algebra $\sdiff(\Sigma)$ of volume preserving diffeomorphims of some associated manifold $\Sigma\,$ (see, for example, \cite{MN}). Thus, for example, the original 4D self-duality equation may be interpreted as a two-dimensional $\sigma$-model \cite{Park} with fields taking values in the Lie algebra $\sdiff(\Sigma^2).$ This splitting into two distinct sets of coordinates is mirrored here with the coordinates $z_i$ on the base space $M$ and the fibre coordinates $\theta_i.$ Indeed, here the Lie algebra  $\sdiff(\Sigma^2)$ is replaced by the algebra of Poisson-preserving sympletic vector fields on the algebraic torus $\bT$, and the Lie group $\operatorname{SDiff}(\Sigma)$ by the group of automorphisms of this torus.

Connections between \hk geometry and Donaldson-Thomas theory are also not new. Gaiotto, Moore and  Neitzke  explained a beautiful connection between these subjects \cite{GMN1}, and went on  to describe explicit examples leading to the Hitchin metric on the moduli space of Higgs bundles \cite{GMN2}. The approach of \cite{RHDT2} is inspired by, and closely-related to, their work, although it is strictly different, since it deals with the `confomal limit', and hence leads to complex \hk structures rather than real ones.

As will be explained below, and as explained in \cite{RHDT2}, the function $W$ has also to satisfy certain homogeneity conditions in order for it to define a Joyce structure, and geometrically these imply the existence of a conformal (actually a homothetic) Killing vector on the \hk manifold $X=\cT_M$. In 4-dimensions, \hk manifolds with such a conformal Killing vector were studied by Dunajski and Tod \cite{DT}, with the corresponding Einstein-Weyl structures on the orbit space being constructed, following Hitchin, and the associated mini-twistor space constructed. Such constructions will obviously extend to arbitrary dimensional \hk manifolds with a conformal Killing vector, but the details do not appear to have been explicitly written down.

As well as bringing together some of the disparate sources, which span 50 years of research in several different areas, the reason for writing the paper is to give a precise geometric description of  the structure on stability space one expects to be encoded by DT theory. This is a complex \hk structure on the tangent bundle, but with certain extra features which it seems worthwhile making explicit. We also took the opportunity to give a geometric characterisation of the \hk structures arising from the above construction.

\subsection*{Notation and terminology}
We use the notation $\cT_M$ to  denote the holomorphic tangent bundle  of a complex manifold $M$. Given a holomorphic map of complex manifolds $\pi\colon X\to M$, a tangent vector $v\in \cT_{X,x}$ at a point $x\in X$ will be  called vertical  if $\pi_*(v)=0\in \cT_{M,\pi(x)}$, and  a holomorphic vector field  on $X$ will be called vertical if its value at each point  is  vertical.

A system of  co-ordinates $(z_1,\cdots,z_n)$ on a complex manifold $M$ gives natural linear co-ordinates $(\theta_1,\cdots,\theta_n)$ on the tangent spaces $\cT_{M,p}$  by writing a tangent vector in the form $v=\sum_i \theta_i\cdot  \frac{\partial}{\partial z_i}$. We thus obtain a system of co-ordinates $(z_1,\cdots, z_n,\theta_1,\cdots, \theta_n)$ on the total space of the tangent bundle $\cT_M$, which we refer to as being  induced by the co-ordinate system $(z_1,\cdots,z_n)$ on $M$.

\subsection*{Acknowledgements}
The first author is grateful to the Sydney Mathematical Research Institute for support and hospitality while this research was being carried out, and to the Royal Society for financial support. He is also grateful to Dominic Joyce for explaining the twistor construction of Section \ref{twistor}. Both authors would like to thank Maciej Dunajski for his comments on the first version of this paper.


\section{Complex hyperk{\"a}hler manifolds and affine symplectic fibrations}
\label{two}

In this section we introduce the notion of an affine symplectic fibration of  a complex \hk manifolds, and give a local description of an arbitrary complex \hk manifold admitting such a map.

\subsection{Affine symplectic fibrations }

A \emph{complex \hk structure} on  a complex manifold $X$ consists of a non-degenerate symmetric bilinear form 
$g\colon \cT_X \tensor \cT_X \to \O_X,$
together with three endomorphisms $I,J,K\in \End(\cT_X)$, which  satisfy the quaternion relations
\[I^2=J^2=K^2=IJK=-1,\]
and which are covariantly constant with respect to the holomorphic Levi-Civita connection associated to  $g$. By a \emph{complex \hk manifold}  we mean  a complex manifold equipped with a complex \hk structure. 

\begin{remark}\label{real}
Viewing the complex manifold $X$ as consisting of the underlying smooth manifold $X_{\bR}$ equipped with a complex structure $S\in \operatorname{End}(T_{X_{\bR}})$, we can use the identification $T_{X_{\bR}}\tensor_{\bR}\bC\isom \cT_X\oplus \overline{\cT}_X $ to define endomorphisms 
$I_{\bR},J_{\bR},K_{\bR}\in \End(T_{X_{\bR}})$ which commute with $S$ and satisfy the quaternion relations, together with a non-degenerate symmetric bilinear form $g_{\bR}\colon T_{X_{\bR}}\tensor T_{X_{\bR}}\to \bR$ satisfying $g_{\bR}(SX,SY)=-g_{\bR}(X,Y)$. This defines a \hk structure on the smooth manifold $X_{\bR}$ in the usual sense,  except that the metric $g_{\bR}$ is indefinite.
\end{remark}

Given a complex \hk structure on a complex manifold $X$ 
we define   a holomorphic symplectic form $\Omega_I$ on $X$ by the rule $\Omega_I(X,Y)=g(X,I(Y))$. The holomorphic symplectic forms  $\Omega_J$ and $\Omega_K$ are defined analogously.  We also introduce the combinations $\Omega_\pm=\Omega_J\pm i\Omega_K$. These are closed, holomorphic 2-forms, but are not symplectic. Indeed, since\begin{equation}
\label{cricket} g(v_1,(I\pm i)v_2)=g(K(v_1),(J\pm iK)(v_2))=\Omega_{\pm}(K(v_1),v_2),\end{equation}
the kernels of the forms $\Omega_{\pm}$ are precisely the eigenspaces of $I$ with eigenvalues $\mp i$. These two subspaces are half-dimensional, and are exchanged by the action of $J$.

\begin{definition}
 \label{fibered}
 Let $X$ be a complex \hk manifold. A holomorphic map $\pi\colon X\to M$ is an \emph{affine symplectic fibration}  if there exists a holomorphic symplectic form $\omega$ on $M$ such that $\Omega_-=\pi^*(\omega)$.
 \end{definition}
 
We call the symplectic form $\omega$ the \emph{base symplectic form} of the affine symplectic fibration.  For an explanation of the name affine symplectic see Remark \ref{name} below.
 
\begin{remark}
\label{verti}
The vertical tangent vectors for $\pi$ are clearly contained in the kernel of  the form $\Omega_-=\pi^*(\omega)$. The assumption that  $\omega$ is non-degenerate ensures that on the open dense subset where $\pi$ is a submersion this inclusion is an equality.  But by \eqref{cricket}, the kernel of $\Omega_-$  also coincides with the $+i$ eigenspace of the operator $I$, and is therefore everywhere half-dimensional. It follows that $\pi$ is  a submersion with half-dimensional fibres. \end{remark}

 Let $\cV_\pi\subset \cT_{X}$ denote the bundle of vertical tangent vectors for the map $\pi\colon X\to M$. Consider a nonzero vertical tangent vector  $v\in \cT_{X,x}$  at some point $x\in X$.   Since $J$ switches the two eigenspaces of $I$, the tangent vector $J(v)$ is not vertical. Thus the map $v\mapsto \pi_*(J(v))$  induces an isomorphism $\cV_{\pi,x}\to \cT_{M,\pi(x)}$. Putting these maps together gives a bundle isomorphism \[b\colon \cV_\pi\to \pi^*(\cT_M)\] which we will call the \emph{basing map} of the affine symplectic fibration.

\begin{remark}
\label{name}
 Note that by \eqref{cricket}, the kernel of the form $\Omega_+$ is precisely the $-i$ eigenspace of $I$. Thus $\Omega_+$ restricts to a holomorphic symplectic form on each fibre of the map $\pi\colon X\to M$. 
 The basing map relates this symplectic form to the symplectic form $\omega$ on the base $M$. Indeed, if $v_1,v_2\in \cT_{X,x}$ are vertical tangent vectors, then since $(J-iK)J=J(J+iK)$, we have\[\Omega_+(v_1,v_2)=\Omega_-(J(v_1),J(v_2))=(\pi^*\omega)(J(v_1),J(v_2))=\omega(b(v_1),b(v_2)).\]
 This now gives the reason for the name affine symplectic fibration: any given fibre $\pi^{-1}(m)$ of the map $\pi\colon X\to M$ has the property that each of its tangent spaces is identified via the basing map with the fixed symplectic vector space $\cT_{M,m}$.
 \end{remark}

\subsection{Standard example}
\label{eg}Let $n\geq 2$ be an even integer, and take $M=\bC^n$. Let $X$ denote the total space of the holomorphic tangent bundle $\cT_M$, with its canonical projection  $\pi\colon X\to M$. Let  $(z_1,\cdots,z_n)$ be standard linear co-ordinates on  $M$, and let $(z_1, \cdots, z_n,\theta_1,\cdots,\theta_n)$ be the induced co-ordinate system on $X$, as explained in the notation and terminology section.

Choose a  $n\times n$ non-degenerate, skew-symmetric matrix $\omega_{pq}$ and introduce the holomorphic symplectic form
\[\omega = \sum_{p,q} \omega_{pq} \cdot dz^p \wedge dz^q,\]
on $M$.
Let $X^0\subset X$ be an open subset, and  
let $W\colon X^0 \to \bC$ be a holomorphic function  satisfying the partial differential equations
\begin{equation}
\label{pde}
\frac{\partial}{\partial \theta_k} \bigg(\frac{\partial^2 W}{\partial \theta_i \partial z_j }-\frac{\partial^2 W}{\partial \theta_j  \partial z_i }-\sum_{p,q} \eta^{pq} \cdot \frac{\partial^2 W}{\partial \theta_i \partial \theta_p} \cdot \frac{\partial^2 W}{\partial \theta_j \partial \theta_q}\bigg)=0,\end{equation}
where $\eta^{ij}$ is the inverse matrix to $\omega_{ij}$. Then define vector fields
\begin{equation}
\label{above}v_i= \frac{\partial}{\partial \theta_i}, \qquad h_i= \frac{\partial}{\partial z_i} + \sum_{p,q} \eta^{pq} \cdot \frac{\partial^2 W}{\partial \theta_i \partial \theta_p} \cdot \frac{\partial}{\partial \theta_q}.\end{equation}

We define a complex \hk structure on $X^0$ by setting 
\begin{equation}
\label{ijk1}I(v_i)=i\cdot v_i,   \qquad  J(v_i)=h_i,   \qquad K(v_i)=-i h_i, \end{equation}
\begin{equation}
\label{ijk2}I(h_i)=-i\cdot h_i, \qquad J(h_i)=-v_i,  \qquad K(h_i)=-i v_i.\end{equation}
\begin{equation}
\label{g}g(v_i,v_j)=0, \qquad g(v_i,h_j)=\omega_{ij}, \qquad g(h_i,h_j)=0.\end{equation}
The quaternion relations are immediate, as is the fact that $I$, $J$ and $K$ preserve the metric. It is clear that  $[v_i,v_j]=0$, and 
the  equation \eqref{pde}  implies that $[h_i,h_j]=0$. Note that
\[g([v_i,h_j],h_k)=\frac{\partial^3 W}{\partial \theta_i \partial \theta_j \partial \theta_k}.\]
Using the Koszul formula to compute the Levi-Civita connection, most terms vanish, and we find that
\begin{equation}
\label{blah}\nabla_{h_i} (h_j)=-\sum_{p,q} \eta^{pq}\cdot \frac{\partial^3 W}{\partial \theta_i \partial \theta_j \partial \theta_p}\cdot h_q, \qquad \nabla_{h_i} (v_j)=-\sum_{p,q} \eta^{pq} \cdot \frac{\partial^3 W}{\partial \theta_i \partial \theta_j \partial \theta_p}\cdot v_q,\end{equation}
and $\nabla_{v_i} (h_j)=0= \nabla_{v_i} (v_j)$.
It is then easy to check  that $\nabla$ preserves $I$, $J$ and $K$.

The associated holomorphic  symplectic forms are
\[\Omega_I=-i \cdot \sum_{p,q} \omega_{pq} \cdot v^p\wedge h^q, \]\[ \Omega_J=\frac{1}{2}\cdot \sum_{p,q} \omega_{pq} (v^p\wedge v^q + h^p \wedge h^q), \qquad \Omega_K= -\frac{i}{2}\cdot \sum _{p,q} \omega_{pq} (v^p\wedge v^q - h^p \wedge h^q) \]
where we used the dual bases of covectors
\[h^j= dz^j, \qquad v^j= d\theta^j + \sum_{r,s} \eta^{jr} \cdot \frac{\partial^2 W}{\partial \theta_r \partial \theta_s} \cdot dz^s\]
so that $(h^j,v_i)=0=(v^j,h_i)$ and $(h^j,h_i)=\delta_{ij}=(v^j,v_i)$. In particular 
\[\Omega_-=\Omega_J-i\Omega_K=\sum _{p,q} \omega_{pq} \cdot h^p \wedge h^q=\pi^*(\omega),\]
which shows that the restriction $\pi\colon X^0\to M$ is an affine symplectic fibration.

\begin{remark}
\label{pemb}
Note that since only the second derivatives of  the function $W$ with respect to the fibre variables $\theta_i$ play any role in the definition of the \hk structure, this function is only well-defined up to the addition of functions which are at most quadratic in the fibre directions. The transformation  of $W$ under symplectic co-ordinate changes $(z_1,\cdots,z_n)\mapsto (w_1,\cdots, w_n)$ on the base $M$ is written out explicitly in \cite[Section 4.2]{A2}. In addition to the obvious substitutions, the function $W$ picks up an extra term which is cubic in the fibre directions.
\end{remark}

\begin{remark}
\label{curv}
A calculation\todo{Is there an easier way to see this?} with the equation \eqref{pde} shows that the curvature component \[R(h_i,h_j)=\nabla_{h_i}\circ \nabla_{h_j}-\nabla_{h_j}\circ \nabla_{h_i}=0,\] and it is immediate that  $R(v_i,v_j)=0$. On the other hand
\begin{equation*}
R(h_j,v_i)(h_k)=\sum_{p,q} \eta^{pq} \cdot \frac{\partial^4 W}{\partial \theta_i \partial \theta_j \partial \theta_k \partial \theta_p}\cdot h_q, \qquad R(h_j,v_i)(v_k)=\sum_{p,q} \eta^{pq} \cdot \frac{\partial^4 W}{\partial \theta_i \partial \theta_j \partial \theta_k \partial \theta_p}\cdot v_q.\end{equation*}
Thus we conclude that the metric is flat precisely if the restriction of the function $W\colon X\to \bC$ to each fibre $\pi^{-1}(m)=\cT_{M,p}$  is a polynomial function of degree $\leq 3$. \end{remark}
 
 \subsection{Normalised affine symplectic fibrations}

 Consider the  total space $X=\cT_M$  of the holomorphic tangent bundle of  $M$, with its natural projection $\pi\colon X\to M$. Denote by $\cV_\pi \subset \cT_X$ the bundle of vertical tangent vectors.
There is a canonical bundle isomorphism
\[\nu\colon \cV_\pi  \to \pi^*(\cT_M)\]
sending a vertical tangent vector $v\in \cT_{X,x}$ to the corresponding tangent vector $\nu(v)\in \cT_{M,\pi(x)}$. 

\begin{definition}\label{based}Let $M$ be a complex manifold. Suppose given a complex \hk structure on an open subset $X^0$ of the total space $X=\cT_M$, for which the restriction of the projection map $\pi\colon X\to M$ is an affine symplectic fibration. Then we call this  affine symplectic fibration  \emph{normalised} if the  basing map $b$ coincides with the natural map $\nu$.
\end{definition}

Note that the examples of Example \ref{eg} are normalised in this sense, since in the co-ordinate system considered there, the map $\nu$ is defined by $\nu(\frac{\partial}{\partial \theta_i})=\frac{\partial}{\partial z_i}$, and it is immediate from the definition that $\pi_*(J(\frac{\partial}{\partial \theta_i}))=\frac{\partial}{\partial z_i}$. Conversely, we have the following:

\begin{prop}
\label{based2}
Let $M$ be a complex manifold, and suppose  given a complex \hk structure on an open subset $X^0$ of the total space $X=\cT_M$, for which the restriction of the projection map $\pi\colon X\to M$ is a normalised affine symplectic fibration. Then, shrinking $X^0$ if necessary, the \hk structure arises via the construction of Section \ref{eg}.
\end{prop}

\begin{proof}
By the holomorphic Darboux theorem, we can find local co-ordinates $z_i$ on the base $M$ such that the base symplectic form can be written in the form $\sum_{i,j} \omega_{ij} \cdot dz^i\wedge dz^j$ for some constant skew-symmetric matrix $\omega_{ij}$. Take the induced co-ordinate system   $(z_i,\theta_j)$ on $X$ and set
$v_i=\frac{\partial}{\partial \theta_i}$ and $h_i=J(v_i)$.
Then,   as in Remark \ref{verti}, since the $v_i$ are vertical tangent vector fields they satisfy $I(v_i)=i v_i$, and it follows that in the basis of vector fields $(v_i,h_j)$ the operators $I,J,K$ are given by the formulae of Section \ref{eg}. 

The normalisation condition is the statement that $\pi_*(h_i)=\frac{\partial}{\partial z_i}$, and we can therefore write
\[h_j = \frac{\partial}{\partial z_j} + \sum_{p,q} \eta^{pq} \cdot c_{jp}(z,\theta) \cdot \frac{\partial}{\partial \theta_q}\]
for some locally-defined holomorphic functions $c_{jp}(z,\theta)$. Since $I$ preserves the metric $g$, the eigenspaces of $I$ are necessarily isotropic. Thus the metric is determined by
\[2g(v_i,h_j)=-g(h_j,(J-iK)(h_i))= \Omega_-(h_i,h_j)=\omega(\pi_*(h_i),\pi_*(h_j))=2\omega_{ij},\]
and therefore coincides with that of  Section \ref{eg}.

Consider next the expression
\[g\big([v_i,h_j],h_k\big)=\frac{\partial}{\partial \theta_i}  c_{jk}(z,\theta).\]
We claim that  the expression on the right is completely symmetric in $i,j,k$. 
To prove this, note first that $[v_i,h_j]$ is a vertical vector field. Since $\nabla$ preserves the eigenspace decomposition of $I$, the relation
$\nabla_{v_i}(h_j)-\nabla_{h_j}(v_i)=[v_i,h_j]$
implies that $\nabla_{v_i}(h_j)=0$, and so
\[g\big([v_i,h_j],h_k\big)=-g\big(\nabla_{h_j}(v_i),h_k\big).\]
Similarly, since $[h_i,h_j]$ is vertical, both sides of the expression
$\nabla_{h_i}(h_j)-\nabla_{h_j}(h_i)=[h_i,h_j]$
must vanish. Since $J$ is covariantly constant, and  $g(h_j,v_k)$ is constant,  we therefore have
\[g\big(\nabla_{h_j}(v_i),h_k\big)=g\big(\nabla_{h_j}(J (v_i)),J(h_k)\big)=-g\big(\nabla_{h_j} (h_i),v_k\big)=g\big(h_i,\nabla_{h_j}(v_k)\big),\]
which together with $\nabla_{h_i}(h_j)=\nabla_{h_j}(h_i)$ proves the required symmetry property.

It now follows that we can write $h_i$ in the form \eqref{above} for some locally-defined function $W=W(z,\theta)$. The relations $[h_i,h_j]=0$ obtained above then imply that $W$ satisfies the partial differential equations \eqref{pde}, which completes the proof.
\end{proof}

\subsection{Developing maps}
In this section we prove that an arbitrary complex \hk manifold with an affine symplectic fibration is locally isomorphic to one arising from the construction of Section \ref{eg}. For this purpose we introduce the following notion:

\begin{definition}
Let $\pi\colon X\to M$ be an affine symplectic fibration.
A \emph{developing map} defined on an open subset $U\subset X$ is  defined to be an open embedding $f\colon U\to \cT_M$, commuting with the projections to $M$, such that 
\begin{equation}
\label{ploop}\nu(f_*(v))=\pi_*(J(v)),\end{equation}
for any vector field $v$ on $U$ which is vertical for the restriction of the map $\pi$.
\end{definition}

\begin{prop}
Let $\pi\colon X\to M$ be an affine symplectic fibration. Then 
a developing map exists in a neighbourhood of any given point $x\in X$.\end{prop}

\begin{proof}
Let $z_i$ be local Darboux co-ordinates on $M$ for the base symplectic structure $\omega$, as in the proof of the previous result. Since $\pi$ is a submersion, we can complete these to local co-ordinates  $(z_i,\phi_i)$ on $X$ so that the vertical tangent spaces are spanned by the vector fields $v_i=\frac{\partial}{\partial \phi_i}$. 
Let us write \[h_i=J(v_i)=\sum_j a_{ij}(z,\phi)\cdot \frac{\partial}{\partial z_j}+\sum_j b_{ij}(z,\phi)\cdot \frac{\partial}{\partial \phi_j}\]
for some locally-defined holomorphic functions $a_{ij}(z,\phi)$ and $b_{ij}(z,\phi)$. 

Consider the induced co-ordinate system $(z_i,\theta_j)$ on the total space $\cT_M$. Then $\nu(\frac{\partial}{\partial \theta_j})=\frac{\partial}{\partial z_j}$, and the defining property of the developing map  $\theta_j=\theta_j(z_i,\phi_i)$ becomes the condition that \[\sum_j \frac{\partial \theta_j}{\partial \phi_i}\cdot \frac{\partial}{\partial z_j} =\sum_j a_{ij}(z,\phi)\cdot \frac{\partial}{\partial z_j},\]
for all $i$. To show the local existence of such a map we must check that $\frac{\partial a_{ij}}{\partial \phi_k}  =\frac{\partial a_{kj}}{\partial \phi_i} $. 
In more intrinsic terms this is the statement that the expression
\[\pi_*([v_k, h_i])=\sum_j \frac{\partial a_{ij}}{\partial \phi_k}\cdot \frac{\partial}{\partial z_j}\]
is symmetric under exchanging $i$ and $k$. Since $\nabla$ preserves the eigenspaces of $I$, and hence the sub-bundle of vertical vector fields, the relation
$[v_k,h_i]=\nabla_{v_k}(h_i)-\nabla_{h_i}(v_k)$
implies that
\[\pi_*([v_k,h_i])=\pi_*(\nabla_{v_k}(h_i)).\] Using the fact that $J$ is covariantly constant we get
\[\nabla_{v_k}(h_i)-\nabla_{v_i}(h_k)= J\big(\nabla_{v_k}(v_i)-\nabla_{v_i}(v_k)\big)=J([v_i,v_k])=0,\]
which gives the required symmetry. It is easy to see from the relation \eqref{ploop} that the derivative of the developing map is an isomorphism, so restricting its domain if necessary we can assume that it is a local embedding. 
\end{proof}

The next result gives the promised  local description of complex \hk manifolds with an affine symplectic fibration.
\begin{prop}
 Let $\pi\colon X\to M$ be a complex \hk manifold equipped with an affine symplectic fibration. Take a point $x\in X$ and let $f\colon U\to \cT_M$ be a developing map  defined on an open neighbourhood $x\in U\subset X$. Set $V=f(U)\subset \cT_M$, and use the resulting isomorphism $f\colon U\to V$ to transfer the complex \hk structure from $U$ to $V$. Then, after possibly shrinking $U$, the \hk structure on $V$ arises from  the construction of Section \ref{eg}. \end{prop}

\begin{proof}
 Since the developing map commutes with the projections to $M$ it is clear that the restriction of the projection $\pi\colon \cT_M\to M$ is an affine symplectic fibration for the transferred \hk structure on $V$. By definition of the developing map it is moreover normalised. The result therefore follows from Proposition \ref{based2}.
\end{proof}


\section{Joyce structures}
\label{three}

In this section we  give a more geometric definition of the notion of a Joyce structure from \cite{RHDT2}, at least in the case when the form $\eta$ appearing there is non-degenerate. 
A Joyce structure is a rather complicated combination of geometric structures which one expects to find on the space of stability conditions on a CY$_3$ triangulated category. The simplest part of this structure, explained in Section \ref{prejoyce}, is obtained from  the basic properties of spaces of stability conditions, in particular the existence of a local isomorphism to a vector space. The rest of the data  of a Joyce  structure, discussed in Section \ref{joyce}, is  induced  by the  Donaldson-Thomas invariants of the category in a rather indirect and involved way which is explained in \cite{RHDT2} and briefly summarised in the Appendix.

\subsection{Spaces with period maps}
\label{prejoyce}

Let $M$ be a complex manifold and $\cH$ a holomorphic vector bundle on $M$. By a \emph{bundle of lattices} in $\cH$ we mean a  holomorphically-varying collection  of subgroups $\cL_p\subset \cH_p$ in the fibres    of $\cH$ such that the induced maps $\cL_p\tensor_{\bZ} \bC\to \cH_{p}$ are all isomorphisms. Any such bundle of lattices determines a locally-constant subsheaf $\cL\subset \cT_M$ consisting of the holomorphic sections of $\cT_{M}$ whose values at each point $p\in M$ lie in the subgroup $\cL_p$. This local system in turn determines  a unique flat connection  on the bundle $\cH$, whose flat sections are precisely the $\bC$-linear combinations of the sections of the subsheaf $\cL$. 

By a \emph{period map} on a complex manifold $M$ we mean a  local isomorphism
\begin{equation}\label{per}\varpi\colon M\to \Hom_{\bZ}(\Gamma,\bC),\end{equation}
where  $\Gamma\isom \bZ^{\oplus n}$ is a free abelian group of finite rank.  Thus $\varpi$ is a holomorphic map from $M$ to the vector space $\Hom_{\bZ}(\Gamma,\bC)$ whose derivative
\begin{equation}\label{potty}(D\varpi)_p\colon \cT_{M,p}\to \Hom_{\bZ}(\Gamma,\bC),\end{equation} at each point $p\in M$ is an isomorphism.   The inverse images of the subgroup $\Gamma^*=\Hom_{\bZ}(\Gamma,\bZ)$  under the maps \eqref{potty} define a bundle of lattices in $\cT_M$, and  we therefore obtain a flat   connection $\nabla$. The period map $\varpi$ now gives  a distinguished vector field $Z\in \Gamma(M,\cT_M)$ via the assignment \[p\mapsto Z_p=(D\varpi)_p^{-1}(\varpi(p))\in \cT_{M,p}.\]

Let us express all this in co-ordinates. Take a basis $(\gamma_1,\cdots,\gamma_n)$ for the lattice $\Gamma$. We obtain  holomorphic functions  \[z_i\colon M\to \bC, \qquad z_i(p)=\varpi(p)(\gamma_i),\] and the assumption that $\varpi$ is a local isomorphism ensures that $(z_1,\cdots,z_n)$  form a system of co-ordinates on $M$.  The dual of the derivative of $\varpi$  at a point $p\in M$ sends $\gamma_i\in \Gamma$ to the element $dz^i\in \cT_{M,p}^*$. This implies that the connection $\nabla$ is torsion-free and that the $z_i$ are flat co-ordinates for this connection. 
The distinguished vector field is  $Z=\sum_i z_i\cdot \frac{\partial}{\partial z_i}$. Note that it follows that  $\nabla(Z)=\id$.

In general we would like to consider complex manifolds $M$ with a locally-defined period map $\varpi$. Such spaces arise for example as discrete quotients of manifolds with a globally-defined period map.  We therefore make the following

\begin{defn}
\label{pp}
A \emph{period structure} on a complex manifold $M$  consists of data
\begin{itemize}
\item[(P1)] a bundle of lattices $\cL\subset \cT_M$ whose associated flat connection we denote by $\nabla$;

\item[(P2)]  a distinguished vector field $Z\in \Gamma(M,\cT_M)$ satisfying $\nabla(Z)=\id$.
\end{itemize}
\end{defn}

Take a point $p\in M$ and set $\Gamma^*=\cL_p$.  A basis of the free abelian group $\Gamma^*$ extends uniquely to a basis of $\nabla$-flat sections $\phi_1,\cdots,\phi_n$ of the tangent bundle $\cT_M$ over a contractible open neighbourhood $p\in M^0\subset M$. Writing the vector field $Z$ in the form $Z=\sum_i z_i\cdot \phi_i$ then defines holomorphic functions $z_i\colon M^0\to \bC$, and condition (P2) then implies that $\phi_i=\frac{\partial}{\partial z_i}$.  It follows that on the open neighbourhood $p\in M^0\subset M$ the structure arises from a period map as explained above. Note in particular that the connection $\nabla$ is necessarily torsion-free.
\smallskip

In the situations of interest below we consider period maps \eqref{per} in which the lattice $\Gamma$ has a natural skew-symmetric integral form $\eta\colon \Gamma\times\Gamma\to \bZ$. The local version of this is

\begin{defn}
\label{pps}
A \emph{period structure with skew-form} consists of a period structure as above, together with a $\nabla$-flat skew-symmetric form
\[\eta\colon \cT_M^*\times\cT_M^*\to \O_M,\]
which takes  integral values on the lattices $\cL^*\subset \cT_M^*$.
\end{defn}

We will be particularly interested in the case when the form $\eta$ is non-degenerate. The inverse then defines a complex symplectic form
 \[\omega\colon \cT_M\times\cT_M\to \O_M,\]
 taking rational values on the lattices $\cL\subset \cT_M$.

\subsection{Joyce structures}
\label{joyce}

We can now give a geometric definition which is a slight strengthening of that  of a Joyce structure given in \cite{RHDT2}. We discuss the differences in Remark \ref{diff} below.

\begin{defn}
Let $M$ be a complex manifold, and let $X$ denote the total space of the holomorphic tangent bundle. A \emph{strong Joyce structure} on $M$ consists of
\begin{itemize}
\item[(a)] a period structure with skew-form $(\cL,Z,\eta)$ on $M$;
\item[(b)] a  complex \hk structure $(g,I,J,K)$ on $X$, \end{itemize}
satisfying the following conditions
\begin{itemize}
\item[(J1)] the canonical projection $\pi\colon X\to M$ is a normalised  affine symplectic fibration;

\item[(J2)] the base symplectic structure $\omega\in \Gamma(M,\bigwedge^2 \cT_M)$ is the inverse to the form $\eta$;

\item[(J3)] the involution $\iota\colon X\to X$ acting by $-1$  on the fibres of $\pi$ satisfies
\vspace{-.15em}
\[\iota^*(g)=-g, \qquad \iota^*(I)=I, \qquad \iota^*(J\pm iK)=-(J\pm iK);\]
\item[(J4)] the $\nabla$-horizontal lift\todo{Explain?} $E\in \Gamma(X,\cT_X)$ of the vector field $Z$ satisfies
\vspace{-.15em}
\[\Lie_E(g)=g, \qquad \Lie_E(I)=0, \qquad \Lie_E(J\pm iK)=\mp(J\pm iK);\]
\item[(J5)] the \hk structure is invariant under translations by the lattice $(2\pi i)\cdot \cL\subset \cT_M$.
\end{itemize}
\end{defn}

As in Section \ref{prejoyce}, we can take $\nabla$-flat local co-ordinates $z_i$ on $M$ such that the lattice $\cL\subset \cT_{M,p}$ is spanned by the vector fields $\frac{\partial}{\partial z_i}$. Let $(z_i,\theta_j)$ be the induced co-ordinate system on $X$.  
 Condition (J2) together with Definition \ref{pps} ensure that the  symplectic form $\omega$ is covariantly constant and can therefore be written in the form $\omega=\sum_{i,j} \omega_{ij} \cdot dz_i\wedge dz_j$, with $\omega_{ij}$ a  non-degenerate, skew-symmetric matrix.  Moreover,  the inverse matrix  $\eta^{ij}$ is integral.  Condition (J1) and  Proposition \ref{based2} then imply that the complex \hk structure $(g,I,J,K)$ arises from the formulae of Example \ref{eg} for some locally-defined holomorphic function $W\colon X\to \bC$  satisfying the equations
 \begin{equation}
 \label{hoj}\frac{\partial}{\partial \theta_k} \bigg(\frac{\partial^2 W}{\partial \theta_i \partial z_j }-\frac{\partial^2 W}{\partial \theta_j  \partial z_i }-\sum_{p,q} \eta^{pq} \cdot \frac{\partial^2 W}{\partial \theta_i \partial \theta_p} \cdot \frac{\partial^2 W}{\partial \theta_j \partial \theta_q}\bigg)=0.\end{equation}
 
 As in Remark \ref{pemb}, the function $W=W(z,\theta)$ is only well-defined up to the addition of functions which are at most quadratic in the $\theta_i$ variables. 
 Condition (J3) is  equivalent to the statement that it may be taken to be an odd function of the $\theta_i$ co-ordinates.  Definition \ref{pp} shows that $Z= \sum_i z_i \cdot \frac{\partial}{\partial z_i}$, and the $\nabla$-horizontal lift  $E$ is given by the same formula. Condition (J4) then becomes the statement that $W$ can be taken to be homogeneous of degree $-1$ under simultaneous rescaling of the variables $z_i$. Finally, condition (J5) is equivalent to the statement  that the second derivatives of $W$ with respect to the fibre variables $\theta_j$ are  invariant under transformations of the form $\theta_j\mapsto \theta_j+2\pi i$.
 
We can in fact assume that $W$ satisfies the simplified equation appearing in \cite{RHDT2}:
\begin{equation}
\label{simples}\frac{\partial^2 W}{\partial \theta_i \partial z_j }-\frac{\partial^2 W}{\partial \theta_j  \partial z_i }=\sum_{p,q} \eta^{pq} \cdot \frac{\partial^2 W}{\partial \theta_i \partial \theta_p} \cdot \frac{\partial^2 W}{\partial \theta_j \partial \theta_q}.\end{equation}
Indeed, if we  replace $W$ by the expression
\[W(z_i,\theta_j)-\sum_k \theta_k\cdot \frac{\partial W}{\partial \theta_k}(z_i,0),\]
then each of the first two terms on the left-hand side of  \eqref{simples} vanishes along the locus where all $\theta_i=0$. But the third term also vanishes along this locus, because  $W$ is odd in the  co-ordinates $\theta_j$.  Since the equations \eqref{hoj} state that  the left-hand side of  \eqref{simples} is independent of the  co-ordinates $\theta_j$, after the above modification it must  vanish identically.

\begin{remark}
 \label{diff}
  There are three differences between the above definition of a strong Joyce structure and the definition of a Joyce structure in \cite{RHDT2}.
 \begin{itemize}
 \item[(i)] The above formulation assumes that the form $\eta$ is non-degenerate. This is not required in the definition of \cite{RHDT2}. To include degenerate forms in the framework of this paper we would have to consider generalisations of the notion of a \hk structure which have the endomorphisms $I,J,K$, but in which the metric is dropped and replaced with a symmetric, bilinear form on the cotangent bundle.\smallskip
 \item[(ii)]  In \cite{RHDT2} it was only assumed that the third, rather than the second, derivatives of the function $W$ in the fibre directions were periodic. This was to allow the inclusion of certain examples obtained by doubling spaces of stability conditions for which the form $\eta$ is degenerate, and even vanishing. 
 \smallskip
 \item[(iii)] The definition in \cite{RHDT2} allows the function $W$ to be meromorphic. Of course, we could also introduce meromorphic complex \hk structures to deal with this generalisation.
 \end{itemize}
 \end{remark}

Let us briefly consider the linearisation procedure of \cite[Section 7]{RHDT2}. Identifying $M$ with the zero-section $M\subset X$, the tangent bundle $\cT_M$ becomes a sub-bundle of the restriction $\cT_X|_M$.  The fact that $W$ is an odd function of the $\theta_i$ co-ordinates ensures that along the zero section $M\subset X$ we have $h_i=\frac{\partial}{\partial z_i}$. It follows that the Levi-Civita connection preserves the tangents to the zero section and hence induces a torsion-free connection $\nabla^J$ on $\cT_M$. Moreover, Remark \ref{curv} together with the oddness of the function $W$ ensure that this  connection is flat.
From equation \eqref{blah} we see that it is given explicitly by
\begin{equation}
\label{linearjoyce}\nabla_{\frac{\partial}{\partial z_i}} \Big(\frac{\partial}{\partial z_j}\Big)=-\sum_{p,q} \eta^{pq}\cdot \frac{\partial^3 W}{\partial \theta_i \partial \theta_j \partial \theta_p}\Big|_{\theta=0}\cdot \frac{\partial}{\partial z_q}.\end{equation}
This is what was referred to as the linearised Joyce connection in \cite{RHDT2}.

\subsection{Twistor space}
\label{twistor}

We conclude by making some brief remarks on the twistor space of a complex \hk manifold $X$. The existence of such a twistor space is not obvious - the quotienting out by a distribution could potentially result in an object that is not Hausdorff. However, the fact that the space of complex null geodesics is Hausdorff was proved in \cite{LeBrun}. Consider the quadric in the complex projective plane
\[Q=\big\{[a:b:c]\in \bP^2: a^2+b^2+c^2=0\big\}.\]
For each point $q=[a:b:c]\in Q$ the kernel of the operator $aI+bJ+cK$ defines a half-dimensional sub-bundle $\cH(q)\subset \cT_X$. The fact that $I,J,K$ are flat for the Levi-Civita connection $\nabla$ ensures that this sub-bundle  is involutive, since if $v,w$ are sections of $\cH(q)$ then so is $[v,w]=\nabla_v(w)-\nabla_w(v)$. Thus there is a  corresponding foliation $\cF(q)$ of the space $X$ by complex submanifolds. 
We define the twistor space $Z$ to  the set of pairs $(q, \cL)$, where $q\in Q$, and $\cL$ is a leaf of the foliation $\cF(q)$. There is an obvious projection $\pi\colon Z\to Q\isom \bP^1$. Each point $x\in X$ then defines a section of this projection by sending a point $q\in Q$ to the unique leaf of the foliation $\cF_q$ containing the point $x\in X$.\todo{I find it strange that this construction doesn't seem to appear anywhere in the literature, at least beyond the 4-dimensional case. Also, I couldn't find a discussion about what assumptions would be needed to ensure that $Z$ is a complex manifold. That seems to be quite a subtle question, and not at all automatic.}

Note that if $\pi\colon X\to M$ is an affine symplectic fibration, the sub-bundle $\cH(q)\subset \cT_X$ corresponding to the point $q=[0:1:-i]\in Q$ is the vertical sub-bundle defined by the map $\pi$, and so the corresponding fibre of the twistor space $Z_q=\pi^{-1}(q)$ is naturally identified with $M$.


If we  identify $\bP^1$ with the quadric $Q$  via the isomorphism
\[[s:t]\mapsto \big[2st, (s^2-t^2), i(s^2+t^2)\big],\]
then the sub-bundle $\cH(q)$ becomes the kernel of the operator
\[2st I + s^2 (J+iK)- t^2(J-iK).\]
Given a local basis of vector fields $v_i$ and $h_j$ on $X$ in which the operators $I,J,K$ take the simple form \eqref{ijk1} - \eqref{ijk2}, this sub-bundle  is spanned by the vector fields $sv_i + t h_i$.

\begin{remark}
In the context of Section \ref{eg} the sub-bundle $\cH(q)$ is spanned by
\begin{equation}
\label{flows} \epsilon^{-1}\cdot \frac{\partial}{\partial \theta_i} + \frac{\partial}{\partial z_i} + \sum_{p,q} \eta^{pq} \cdot \frac{\partial^2 W}{\partial \theta_i \partial \theta_p} \cdot \frac{\partial}{\partial \theta_q},\end{equation}
where we set $\epsilon=t/s$. 
In \cite{RHDT2} these sub-bundles $\cH(\epsilon)$  played a central role, and were viewed as defining a pencil of Ehresmann connections on the map $\pi \colon X\to M$.

It seems interesting to try to relate the Riemann-Hilbert problems considered in \cite{RHDT2} to the geometry of the twistor space $Z$. 
Note in particular, that the solutions to the Riemann-Hilbert problem at a particular value $\epsilon\in \bC^*$ are annhilated by the flows \eqref{flows}, and are therefore constant on the leaves of the corresponding foliation $\cF(q)$. Thus such solutions naturally define functions on the fibre $Z_q=\pi^{-1}(q)$ of the twistor space. In the context of categories defined by quivers with potential this suggests a relation between these fibres  and the associated cluster Poisson variety.
\end{remark}

\section{Quantum DT invariants and deformations of anti-self-duality}

In \cite{S92} an integrable deformation of Pleba\~nski's first heavenly equation was constructed using ideas from the deformation quantization programme. The nonlinear terms in both the first and second Pleba\~nski equations comes from the Poisson bracket
\[
\{f,g\}=\sum_{i,j} \eta^{ij} \frac{\partial f}{\partial\theta_i} \frac{\partial g}{\partial\theta_j}
\]
and deformations of this bracket - preserving the Jacobi identity - date back to the work of Moyal who introduced such a structure in his phase space approach to quantum mechanics. The Moyal bracket can be defined in terms of the operator 
\[
P= \exp\left[ 
\frac{i\hbar}{2} \eta^{ij} \frac{{ \leftarrow \atop {\displaystyle\partial}~}}{\partial\theta_i}  \frac{{ \rightarrow \atop {\displaystyle\partial}~}}{\partial\theta_j} 
\right]
\]
(where the arrows show which direction the derivatives are to be taken) and with this one defines the non-commutative but associative product $f*g=fPg\,$ and the corresponding Moyal bracket
\[
\{f,g\}_M = \frac{f*g-g*f}{i\hbar}\,.
\]
Since 
\[
\lim_{\hbar\rightarrow 0} \{f,g\}_M=\{f,g\}
\]
this is a deformation of the normal Poisson bracket with $\hbar$ being the deformation parameter. The Jacobi identity for the Moyal bracket follows trivially from the associative property of the underlying $*$-product.

One can also introduce differential operators ${\hat X}_f$ which contain higher derivatives (formally, as these are to all orders in $\hbar$) that are deformations of Hamiltonian vector fields, and have the property
\[
[{\hat X}_f,{\hat X}_g] = {\hat X}_{\{f,g\}_M}
\]
where the left-hand side is commutator of operators. These operator no longer have a clear geometric interpretation but one can use them instead of Hamiltonian vector fields in Lax equations. Applying this idea results in the Moyal deformation of the first Pleba\~nski equation 
\[
\left\{ \frac{\partial W}{\partial z_i} ,  \frac{\partial W}{\partial z_j} \right \}_M=1
\]
and the Moyal deformed version of the Pleba\~nski's second heavenly equation
\begin{equation}
\frac{\partial^2W}{\partial z_i \partial \theta_j} - \frac{\partial^2W}{\partial z_j \partial \theta_i}= \left\{ \frac{\partial W}{\partial\theta_i}, \frac{\partial W}{\partial\theta_j} \right\}_M\,.
\label{deformedPleb}
\end{equation}
The integrability of such equations remains, though in a formal setting. In \cite{S92} solutions were constructed as a formal power series in $\hbar$ and in \cite{Tak94} the equations were derived from a Riemann-Hilbert splitting problem in the associated Moyal loop group.

The $*$-product, and hence the Moyal bracket, has a very rich structure. In terms of a torus basis,
\[
e^{\theta(\alpha)} * e^{\theta(\beta)} = (-q^{\frac{1}{2}})^{\<\alpha,\beta\>} e^{\theta(\alpha+\beta)}\,,
\]
where $\<\alpha,\beta\>=\eta^{ij}\alpha_i\beta_j\,,$ and hence
\begin{eqnarray}
\{e^{\theta(\alpha)},e^{\theta(\beta)}\}_M 
&=&\frac{1}{i\hbar} \left[ (-q^{\frac{1}{2}})^{\<\alpha,\beta\>} - (-q^{-\frac{1}{2}})^{\<\alpha,\beta\>} \right]e^{\theta(\alpha+\beta)}\,,\\
&=&\frac{2}{\hbar} \sin \left[ \frac{\hbar \<\alpha,\beta\>}{2}\right]e^{\theta(\alpha+\beta)}\,,
\end{eqnarray}
where $q= e^{i(\hbar+2\pi)}$ (the $e^{2\pi i}$ term is to ensure the correct sign for the square root). This algebra has a number of names: the sine-algebra \cite{FZ89} and the quantum torus algebra \cite{KS}. Since
\[
\lim_{\hbar\rightarrow 0} \{e^{\theta(\alpha)},e^{\theta(\beta)}\}_M = \<\alpha,\beta\> e^{\theta(\alpha+\beta)}\,.
\]
in the classical limit the quantum torus algebra becomes the twisted torus algebra.

With the same ansatz 
\[
W=\sum_\alpha F_\alpha(z) \frac{e^{\theta(\alpha)}}{z(\alpha)}
\]
with $F_\alpha$ of degree zero, one obtains the isomonodromy equation
\[
d F_\gamma = \sum_{\alpha+\beta=\gamma}  \frac{1}{i\hbar} \left\{ \mathbb{L}^{\frac{1}{2} \<\alpha,\beta\>} - \mathbb{L}^{-\frac{1}{2} \<\alpha,\beta\>}\right\}
F_\alpha F_\beta d\log z(\beta)\,.
\]
where $\mathbb{L} = e^{i\hbar}\,.$ The construction of solutions to such equations is still at a very early stage, and even in the simplest case, it involves highly intricate calculations \cite{BBS19}. But it appears that the quantum DT invariants play the same role in the construction of solutions to the deformed equation (\ref{deformedPleb}) as DT invariants play in the construction of solutions to equation (\ref{point_intro}).

With the replacement of Hamilton vector fields with infinite order differential operators the link with geometry becomes less clear. At the level of Lax pairs, the replacement of Hamiltonian vector fields with operators seems minor, but these are infinite-order operators, and so the resulting differential equations are of infinite order. Also, any notion of \hk geometry is lost - though it is intriguing to speculate where there is a notion of $q$-deformed \hk structures and $q$-deformed twistor spaces (see, for example, \cite{KKO01,S97}).

\appendix
\section{Summary of \cite{RHDT2}}

For the convenience of the reader we give a brief summary of  the main ideas of \cite{RHDT2}. The key point is that the Donaldson-Thomas (DT) invariants of a CY$_3$ triangulated category can be interpreted as Stokes data for a family of connections over the space of stability conditions. These connections take values in an infinite-dimensional  group $G$ of automorphisms of the space $(\bC^*)^n$ equipped with a constant Poisson structure.  
\subsection{Stokes data}
\label{stokes}
We will start by considering Stokes data in the context of  the finite-dimensional group $\GL_n(\bC)$. Our treatment is based on that in \cite{VTL}, to which we refer the reader for references to the original literature.
 Set $G=\GL_n(\bC)$ and $\fg=\mathfrak{gl}_n(\bC)$, and let $\fh\subset \fg$ denote the Cartan subalgebra of diagonal matrices. Introduce the standard root decomposition
 \[\fg=\fh\oplus \fgod, \qquad \fgod=\bigoplus_{\alpha\in \Phi} \fg_\alpha,\qquad \Phi=\{e_i^*-e_j^*\}\subset \fh^*.\]
 Let us consider a meromorphic connection on the trivial $G$-bundle over $\bP^1$ of the form
\begin{equation}
\label{equation}\nabla=d-\bigg(\frac{U}{\epsilon^2}+\frac{V}{\epsilon} \bigg),\end{equation}
where  $U,V\in\fg $ are constant matrices such that

\begin{itemize}
\item[(i)] $U=\operatorname{diag}(u_1,\cdots, u_n)\in \hreg$ is diagonal with distinct eigenvalues,
\smallskip

\item[(ii)] $V\in \fgod$ has zeroes on the diagonal.
\end{itemize}
\smallskip

 The connection $\nabla$ has a regular singularity at $\epsilon=\infty$, but the singularity at $\epsilon=0$ is irregular, and is essentially the simplest example of such a singularity.   It is a classical fact that the gauge equivalence of a connection in a neighbourhood of an irregular singularity   is not determined solely by the monodromy: we must also consider its Stokes data.

For any ray $r=\bR_{>0}\cdot z\subset \bC^*$ we denote by $\bH_r\subset \bC^*$ the half-plane centered on it. 
The \emph{Stokes rays} of the connection \eqref{equation} at $\epsilon=0$  are defined to be the rays  \[\bR_{>0}\cdot 
(u_i-u_j)=\bR_{>0} \cdot U(\alpha)\subset \bC^*, \qquad \alpha=e_i^*-e_j^*.\]
We then have the following fundamental existence result \cite{BJL}:

\begin{theorem}[Balser, Jurkat, Lutz]
\label{bjl}
 For any non-Stokes ray $r\subset \bC^*$  there is a unique 
flat section  $Y_r\colon \bH_r\to G$ of the connection \eqref{equation}
such that \[Y_r(\epsilon)\cdot \exp(U/\epsilon)\to 1 \text{ as }\epsilon\to 0.\]
\end{theorem}

Using this result we can associate to each Stokes ray $\ell$ a \emph{Stokes factor}
\[\bS_\ell=Y_{r_+}(\epsilon)\cdot Y_{r_-}(\epsilon)^{-1}\in \exp\big(\bigoplus_{U(\alpha)\in \ell} \fg_\alpha\big) \subset G,\]
where $r_+$ and $r_-$  are small clockwise and anti-clockwise perturbations of the ray $\ell$.

Let us now consider varying the diagonal matrix $U\in\fh^{\rm reg}$. It turns out that  we can  uniquely deform the matrix $V\in\fgod$ so that the Stokes factors  $\bS(\ell)$ remain constant. Such deformations are called \emph{isomonodromic} or \emph{iso-Stokes}, and the variation of $V=V(U)$ is described by the partial differential equation
\begin{equation}
\label{isomon}d\log V_\gamma =\sum_{\alpha+\beta=\gamma} [V_\alpha,V_\beta] \cdot d\log U(\beta), \qquad V=\sum_{\gamma\in \Phi} V_\gamma\in\fgod.\end{equation}

In fact the isomonodromy condition stated above is not sufficiently precise. As $U\in \fh^{reg}$   varies, the Stokes rays $\bR_{>0} \cdot U(\alpha)$ may cross, at which point the statement that the Stokes factors are constant ceases to make sense. The correct condition to impose is that for  
any convex sector $\Delta\subset \bC^*$, the clockwise product
\[\bS_p(\Delta)=\prod_{\ell\in \Delta} \bS_p(\ell)\in G,\]
of Stokes factor associated to rays in the sector should be constant as $U\in \hreg$ varies, providing  that no Stokes ray crosses the boundary of $\Delta$.

Let us suppose given the matrix $U\in \hreg$ and the Stokes factors $\bS(\ell)\in G$, and consider the inverse problem of reconstructing the connection \eqref{equation}. To  do this we can first try to construct the canonical half-plane solutions $Y_r(t)$ and then differentiate to obtain the operator $V$. This leads to the following 
Riemann-Hilbert boundary-value problem:
\medskip

\noindent {\bf Riemann-Hilbert problem.} 
For each non-Stokes ray $r\subset \bC^*$ find a holomorphic function
$Y_r\colon \bH_r\to G$
such that the following three properties hold:
\begin{itemize}
\item[(RH1)] $Y_r(\epsilon) \cdot \exp(U/\epsilon) \to 1$ as $t \to 0$ in $\bH_r$,\smallskip

\item[(RH2)] there exists $k>0$ such that $|\epsilon|^{-k} < \|Y_r(\epsilon)\|<|\epsilon|^k$ as $\epsilon\to \infty$ in $\bH_r$,\smallskip

\item[(RH3)] if $\Delta\subset \bC^*$ is a convex sector with $\partial \Delta=\{r_+\}\cup\{r_-\}$ then
\[Y_{r+}(\epsilon)=Y_{r_-}(\epsilon)\cdot \bS(\Delta) \text{ for } \epsilon\in \bH_{r_+}\cap \bH_{r_-}.\]
\end{itemize}

Note that the canonical half-plane solutions of Theorem \ref{bjl} satisfy (RH1) by definition,   and (RH3) holds by the definition of the  Stokes factors. The condition (RH2) follows from the fact that the equation \eqref{equation} has a regular singularity at $\epsilon=\infty$, so that solutions have moderate growth at this point.

\subsection{Stability conditions and DT invariants}
\label{stab}

Let $\D$ be a $\bC$-linear triangulated category of finite type.  
We assume for simplicity that the Grothendieck group
\[\Gamma:=K_0(\D)\isom \bZ^{\oplus n}\] is free of finite rank. The  expression\[ \big\langle[E],[F]\big\rangle=\sum_i (-1)^i \dim_{\bC} \Hom^i(E,F[i]),\]
defines a bilinear form $\<-,-\>\colon \Gamma\times\Gamma\to \bZ$ known as the \emph{Euler form.}

The data of a stability condition on  $\D$ consists of a group homomorphism $Z\colon \Gamma\to \bC$ called the \emph{central charge}, and  for each $\phi\in \bR$ a full subcategory $\cP(\phi)\subset \D$ whose objects are said to be \emph{semistable} of phase $\phi$.   This data is required to satisfy a simple set of axioms \cite{Stab}. The important fact is then that the space $\Stab(\D)$ of all such stability conditions on the category $\D$ is a complex manifold, and the forgetful map
\begin{equation}
\label{cov}\varpi\colon \Stab(D)\to \Hom_{\bZ}(\Gamma,\bC)\end{equation}
sending a stability condition to its central charge is a local isomorphism of complex manifolds.

To be able to define Donaldson-Thomas (DT) invariants we must assume that the category $\D$ satisfies  the  CY$_3$ condition
\[\Hom_{\D}^i(A,B)\isom \Hom^{3-i}_{\D}(B,A)^*.\]
This implies that the Euler form 
  is skew-symmetric. Thus in the terminology of Section \ref{prejoyce} the manifold $\Stab(\D)$ is naturally equipped with  a period structure with skew-form. Our aim is to use  DT theory to enrich this to a Joyce structure.

Let us fix a stability condition $\sigma\in \Stab(\D)$. Given further assumptions on the pair $(\D,\sigma)$  it is possible to define DT invariants
$\DT_{\sigma}(\gamma)\in \bQ$ for each class $\gamma\in \Gamma$. 
In the simplest case, when there are no strictly-semistable objects of class $\gamma\in \Gamma$, and the moduli stack $\cM^{\sigma}(\gamma)$ of semistable objects is smooth, the  invariant $\DT_\sigma(\gamma)$ coincides up to sign with the Euler characteristic of the coarse moduli space of  $\cM^{\sigma}(\gamma)$ viewed as a complex manifold. The definition in the general case is due to Joyce and Song \cite{JS}, and Kontsevich and Soibelman \cite{KS}. There is an equivalent system of  invariants $\Omega_\sigma(\gamma)\in \bQ$ defined by
\[\DT_\sigma(\alpha)=\sum_{\alpha=k\beta} \frac{1}{k^2}\cdot \Omega_\sigma(\beta),\]
These appear in physics as BPS invariants, and in many cases are known  to be integers.

\subsection{The wall-crossing formula}
\label{wcf}
Continuing with the notation of the previous section, the next step is to consider the dependence of the DT invariants  on the stability condition $\sigma\in \Stab(\D)$.  It turns out that for a fixed class $\gamma\in \Gamma$, the invariant $\DT_\sigma(\gamma)$ is constant in the complement of a collection of real codimension-one submanifolds in the space $\Stab(\D)$, but jumps discontinuously as the stability condition crosses one of these walls. Joyce \cite{J1,JS}, and Kontsevich and Soibelman \cite{KS}, were able to describe this wall-crossing behaviour exactly, in such a way that knowledge of all invariants $\DT_\sigma(\gamma)$ at one point $\sigma\in \Stab(\D)$  determines them at all other points. Even more remarkably, in the formulation of \cite{KS}, the resulting wall-crossing formula is exactly the iso-Stokes condition for a family of differential equations  of the form \eqref{equation}, but with the finite-dimensional group $\GL_n(\bC)$ replaced by an infinite-dimensional group of Poisson automorphisms of the space $(\bC^*)^n$.

To explain this in more detail,  introduce the algebraic torus
\[\bT=\Hom_\bZ(\Gamma,\bC^*)\isom (\bC^*)^n,\qquad 
\bC[\bT]=\bigoplus_{\gamma\in \Gamma} \bC\cdot x_\gamma,\]
whose character lattice is $\Gamma$. The Euler form defines an invariant Poisson structure on $\bT$ given on characters by
 \[\{x_\alpha,x_\beta\}=\<\alpha,\beta\> \cdot x_{\alpha+\beta}.\]
To make connection with the material of Section \ref{stokes} we would like to 
consider the group \[G=\operatorname{Aut}_{\{-,-\}}(\bT)\] of algebraic Poisson automorphisms of the variety $\bT$. Note that the group structure on the torus $\bT$ plays no role in this definition. The corresponding Lie algebra $\fg$ consists of algebraic vector fields on $\bT$ whose flows preserve the Poisson structure. The fact that these objects are infinite-dimensional will make some aspects of the following discussion heuristic. Precise discussions can be found in \cite{RHDT}.

For simplicity we assume that the form $\<-,-\>$ is non-degenerate. There is then a root decomposition
 \begin{equation}
 \label{lie}\fg=\operatorname{Vect}_{\{-,-\}}(\bT)=\fh\oplus \fgod,\end{equation}
 where the  Cartan subalgebra
$\fh\isom \Hom_\bZ(\Gamma,\bC)$
consists of translation-invariant vector fields on $\bT$, and the subspace $\fgod\subset \fg$ consists of Hamiltonian vector fields, and can be identified with the Poisson  algebra of  non-constant algebraic functions on $\bT$:
\[ \fgod=\bigoplus_{\alpha\in \Gamma\setminus\{0\}} \fg_\alpha = \bigoplus_{\alpha\in \Gamma\setminus\{0\}} \bC\cdot x^\alpha.\]

Fix a stability condition $\sigma\in \Stab(\D)$. For  each ray $\ell\subset \bC^*$ we can attempt to define  an automorphism of $\bT$ by taking the time 1 Hamiltonian flow of the corresponding DT generating function viewed as a regular function on $\bT$. The action of this automorphism on characters is
\begin{equation}
\label{st}\bS_\sigma(\ell)^*(x_\beta)=\exp \Big\{ -\hspace{-.6em}\sum_{Z(\gamma)\in \ell}  \DT_\sigma(\gamma)\cdot x_\gamma,-\Big\}(x_\beta)=x_\beta\cdot \prod_{Z(\gamma)\in \ell} (1-x_\gamma)^{\Omega_\sigma(\gamma)
\cdot \<\gamma,\beta\>}.\end{equation}
Since the sum and product here could be infinite, making rigorous sense of this requires further work \cite{KS,RHDT}. Let us call a ray $\ell\subset \bC^*$ active if  $\bS(\ell)$ is not the identity.
The wall-crossing formula can now be stated as follows: for any convex sector $\Delta\subset \bC^*$ the  product
\[\bS_{\sigma}(\Delta)=\prod_{\ell\in \Delta} \bS_\sigma(\ell)\in G\]
is  constant as $\sigma\in \Stab(\D)$ varies, providing the boundary of $\Delta$ remains non-active.

\subsection{Joyce structures}

Comparing the results of the last two subsections we come to the remarkable conclusion that the wall-crossing formula is the isomonodromy condition for a family of meromorphic connections on $\bP^1$  of the form
\begin{equation}
\label{conn}\nabla = d - \bigg( \frac{Z}{\epsilon^2} + \frac{\Ham_F}{\epsilon}\bigg) d\epsilon,\end{equation}
parameterised by the points of $\Stab(\D)$ and taking values  in the group $G=\Aut_{\{-,-\}}(\bT)$ of Section \ref{wcf}. Here $Z$ and $\Ham_F$ are constant elements of the Lie algebra \eqref{lie} such that

\begin{itemize}
\item[(i)] $Z\in \fh$ is the central charge $Z\colon \Gamma\to \bC$, viewed as an invariant vector field on $\bT$,\smallskip

\item[(ii)] $\Ham_F\in\fgod$ is the Hamiltonian vector field of a regular function $F=\sum_{\gamma\in \Gamma} F_\gamma \cdot x^\gamma$.
\end{itemize}

Let us choose a  basis $(\gamma_1,\cdots,\gamma_n)\subset \Gamma$ giving co-ordinates $z_i=Z(\gamma_i)$  on $\Stab(\D)$, and $\theta_i=\theta(\gamma_i)$  on its tangent spaces, which are all identified with the fixed vector space $\Hom_{\bZ}(\Gamma,\bC)$ by the derivative of the period map \eqref{cov}.  Set $\eta^{ij}=\<\gamma_i,\gamma_j\>$.

Let $X$ be the total space of the tangent bundle of $\Stab(\D)$ and define  a holomorphic function $W\colon X\to \bC$ by the formula \[W(z_i,\theta_j)=\sum_{\gamma\in \Gamma^\times} F_\gamma(z_1,\cdots,z_n) \cdot \frac{e^{\theta(\gamma)}}{Z(\gamma)}.\]
The variation of $F=F(Z)$ is controlled by the  isomonodromy equation \eqref{isomon}, which in the Lie algebra \eqref{lie}  takes the form
\[dF_\gamma=\sum_{\alpha+\beta=\gamma} \<\alpha,\beta\> \cdot F_\alpha F_\beta\cdot d\log Z(
\beta).\]
 When written in terms of the Joyce function $W$ this becomes
\begin{equation}
\label{point}
\frac{\partial^2 W}{\partial \theta_i \partial z_j}-\frac{\partial^2 W}{\partial \theta_j \partial z_i }=\sum_{p,q} \eta^{pq} \cdot \frac{\partial^2 W}{\partial \theta_i \partial \theta_p} \cdot \frac{\partial^2 W}{\partial \theta_j \partial \theta_q},\end{equation}
which obviously implies the equation \eqref{pde}. 

Reconstructing the connection \eqref{conn}, and hence the Joyce function $W$, from the Stokes data \eqref{st} involves solving a Riemann-Hilbert boundary-value problem analogous to the one stated in Section \ref{stokes}. In the context of the infinite-dimensional Lie algebra \eqref{lie} these problems were considered in detail in \cite{RHDT,RHDT2}.

\end{document}